\documentclass{amsart}
\usepackage{amsmath}
\usepackage{amsthm}
\usepackage{amsfonts}
\usepackage{amssymb}
\usepackage{mathrsfs}
\usepackage{amscd}
\usepackage{url}

\newcommand{\re}{\mathbb{R}}
\newcommand{\cpx}{\mathbb{C}}

\newcommand{\N}{\mathbb{N}}

\newcommand{\half}{\frac{1}{2}}
\newcommand{\lmd}{\lambda}

\newcommand{\eps}{\epsilon}

\def\af{\alpha}

\newcommand{\sig}{\sigma}
\newcommand{\Sig}{\Sigma}

\newcommand{\reff}[1]{(\ref{#1})}
\newcommand{\pt}{\partial}
\newcommand{\prm}{\prime}

\newcommand{\bdes}{\begin{description}}
\newcommand{\edes}{\end{description}}

\newcommand{\bal}{\begin{align}}
\newcommand{\eal}{\end{align}}

\newcommand{\bnum}{\begin{enumerate}}
\newcommand{\enum}{\end{enumerate}}

\newcommand{\bit}{\begin{itemize}}
\newcommand{\eit}{\end{itemize}}

\newcommand{\bea}{\begin{eqnarray}}
\newcommand{\eea}{\end{eqnarray}}
\newcommand{\be}{\begin{equation}}
\newcommand{\ee}{\end{equation}}

\newcommand{\baray}{\begin{array}}
\newcommand{\earay}{\end{array}}

\newcommand{\bca}{\begin{cases}}
\newcommand{\eca}{\end{cases}}

\newcommand{\bcen}{\begin{center}}
\newcommand{\ecen}{\end{center}}

\newcommand{\bbm}{\begin{bmatrix}}
\newcommand{\ebm}{\end{bmatrix}}

\newcommand{\bmx}{\begin{matrix}}
\newcommand{\emx}{\end{matrix}}

\newcommand{\bpm}{\begin{pmatrix}}
\newcommand{\epm}{\end{pmatrix}}

\newtheorem{theorem}{Theorem}[section]

\newtheorem{prop}[theorem]{Proposition}

\newtheorem{lemma}[theorem]{Lemma}

\theoremstyle{definition}

\newtheorem{exm}[theorem]{Example}

\begin{document}

\title{Polynomial Optimization with Real Varieties}

\author[Jiawang Nie]{Jiawang Nie}
\address{Department of Mathematics, University of California San Diego,
9500 Gilman Drive,  La Jolla, CA 92093, USA.}

\email{njw@math.ucsd.edu}
\thanks{The research was partially supported by the NSF grant DMS-0844775.}

\begin{abstract}
We study the optimization problem
\[
\min \quad f(x) \quad s.t. \quad h(x) = 0, \,\, g(x) \geq 0
\]
with $f$ a polynomial and $h,g$ two tuples of polynomials in $x\in \re^n$.
Lasserre's hierarchy is a sequence of sum of squares relaxations
for finding the global minimum $f_{min}$.
Let $K$ be the feasible set.
We prove the following results:
i) If the real variety $V_{\re}(h)$ is finite,
then Lasserre's hierarchy has finite convergence,
no matter the complex variety $V_{\cpx}(h)$ is finite or not.
This solves an open question in Laurent's survey~\cite{Lau}.
ii) If $K$ and $V_{\re}(h)$ have
the same vanishing ideal, then the finite convergence of
Lasserre's hierarchy is independent of the choice of defining polynomials
for the real variety $V_{\re}(h)$.
iii) When $K$ is finite, a refined version of Lasserre's hierarchy
(using the preordering of $g$) has finite convergence.
\end{abstract}



\maketitle

\section{Introduction}

Consider the polynomial optimization problem
\be  \label{pop:gen}
\left\{\baray{rl}
f_{min}:= \min & f(x) \\
s.t. & h_i(x) = 0 \,(i=1,\ldots,m_1),  \\
 & g_j(x) \geq  0 \,(j=1,\ldots,m_2),
\earay \right.
\ee
where $f$ and all $g_i, h_j$ are real polynomials in $x\in\re^n$.
Denote $h:=(h_1,\ldots, h_{m_1})$ and $g:=(g_1,\ldots, g_{m_2})$.
Let $K$ be the feasible set of \reff{pop:gen}.
A standard approach for solving \reff{pop:gen} globally is
{\it Lasserre's hierarchy} of sum of squares (SOS) relaxations \cite{Las01}.
We first give a short review about it.
Let $\re[x]$ be the ring of polynomials with real coefficients
and in variables $x:=(x_1,\ldots,x_n)$.
A polynomial $p$ is SOS if there exist $p_1,\ldots,p_k \in \re[x]$
such that $p=p_1^2+\cdots+p_k^2$.
Denote by $\Sig \re[x]^2$ the set of all SOS polynomials.
A subset $I$ of $\re[x]$ is an ideal if
$I+I \subseteq I$ and $I \cdot \re[x] \subseteq I$.
The tuple $h$ generates the ideal $h_1 \re[x]+\cdots + h_{m_1} \re[x]$,
which is denoted as $\langle h \rangle$.
The $2k$-th {\it truncated ideal} generated by $h$ is
\[
\langle h \rangle_{2k}: =
\left\{
\left. \overset{m_1}{ \underset{ i=1}{\sum} }  \phi_i h_i \right|
\baray{c}
\mbox{ each } \phi_i \in \re[x] \\
\mbox{ and } \deg (\phi_i h_i ) \leq 2k
\earay
\right\},
\]
and the $k$-th {\it truncated quadratic module} generated by $g$ is (denote $g_0= 1$)
\[
Q_k(g):= \left\{
\left. \overset{m_2}{ \underset{ j=0}{\sum} }   \sig_j g_j \right|
\baray{c}
\mbox{each } \sig_j \in \Sig\re[x]^2 \\
\mbox{ and } \deg( \sig_j g_j) \leq 2k
\earay
\right\}.
\]
Let $\N$ be the set of nonnegative integers.
The union $Q(g):=\cup_{k\in \N} Q_k(g)$ is
called the {\it quadratic module} generated by $g$.
Lasserre's hierarchy for \reff{pop:gen} is the sequence of SOS relaxations ($k\in \N$)
\be  \label{sos:Put}
f_k:= \, \max \quad \gamma \quad
s.t. \quad f - \gamma  \in  \langle h\rangle_{2k} + Q_k(g).
\ee
The integer $k$ in \reff{sos:Put} is called a {\it relaxation order}.
The SOS program \reff{sos:Put} is equivalent to a
semidefinite program (SDP) (cf. \cite{LasBok,Lau}).

Next, we describe the dual optimization problem of \reff{sos:Put}.
Let $y$ be a sequence indexed by $\af:=(\af_1,\ldots,\af_n)\in \N^n$
with $|\af| := \af_1+\cdots+\af_n \leq 2k$,
i.e., $y$ is a {\it truncated moment sequence} (tms) of degree $2k$.
Denote by $\mathscr{M}_{2k}$ the space of all tms' whose degrees are $2k$.
Denote by $\lceil a \rceil$ the smallest integer that is not smaller than $a$.
Denote $d_j := \lceil \deg(g_j)/2 \rceil$,
$x^\af := x_1^{\af_1}\cdots x_n^{\af_n}$ and
\[
[x]_t :=
\bbm 1 & x_1 \, \cdots \, x_n & x_1^2  & x_1 x_2 & \cdots  & x_1^t & \cdots &  x_n^t \ebm^T.
\]
For each $k\geq d_j$, expand the product $g_j [x]_{k-d_j}[x]^T_{k-d_j}$ as
\[
g_j [x]_{k-d_j}[x]^T_{k-d_j} =
\sum_{ |\af| \leq 2k}  A^{(k,j)}_\af x^\af,
\]
where each $A^{(k,j)}_\af$ is a constant symmetric matrix. The matrix
\[
L_{g_j}^{(k)}(y) :=
\sum_{ |\af| \leq 2k}  A^{(k,j)}_\af y_\af
\]
is called a {\it localizing matrix}.
For $g_0=1$, $M_k(y):=L_1^{(k)}(y)$ is called a {\it moment matrix}.
The columns and rows of $L_{g_j}^{(k)}(y)$
are indexed by vectors $\af \in \N^n$ with $|\af|\leq k-d_j$.
We refer to Laurent \cite[Section~4]{Lau} for
more details about moment and localizing matrices.
The dual optimization problem of \reff{sos:Put} is (cf. \cite{LasBok,Lau})
\be \label{mom-las:put}
\left\{\baray{rl}
f_k^*:=\underset{ y \in \mathscr{M}_{2k} }{\min} &  \langle f, y \rangle \\
\mbox{s.t.}& L_{h_i}^{(k)}(y) = 0 \,(i=1,\ldots, m_1), \, y_0 = 1, \\
           & L_{g_j}^{(k)}(y) \succeq 0 \,(j=0,1,\ldots, m_2).
\earay \right.
\ee
In the above, $X\succeq 0$ means the matrix $X$ is positive semidefinite.

Let $f_{min}, f_k, f_k^*$, respectively, be the optimal values of
\reff{pop:gen}, \reff{sos:Put} and \reff{mom-las:put}.
It is known that $f_k \leq f_k^* \leq f_{min}$ for all $k$.
The sequences $\{f_k\}$ and $\{f_k^*\}$ are both monotonically increasing.
If $K$ has nonempty interior, then \reff{mom-las:put} has an interior point,
\reff{sos:Put} achieves its optimal value
and $f_k^* = f_k$, i.e., there is no duality gap between
\reff{sos:Put} and \reff{mom-las:put} (cf. \cite{Las01}).
Under the archimedean condition (there exists $R>0$
such that $R-\sum_{i=1}^n x_i^2\in \langle h \rangle + Q(g)$),
Lasserre proved the asymptotic convergence $f_k \to f_{min}$ as $k\to \infty$.
The proof uses Putinar's Positivstellensatz \cite{Put}.
We refer to Lasserre's book \cite{LasBok}, Laurent's survey \cite{Lau}
and Marshall's book \cite{MarBk} for the work in this area.

When $f_k = f_{min}$ occurs for some $k$,
we say that Lasserre's hierarchy has {\it finite convergence}.
An appropriate criterion for checking finite convergence of
$\{f_k\}$ is {\it flat truncation}, as shown in \cite{Nie-ft}.
For the tuple $h$,
define the complex and real algebraic varieties respectively as
\be
V_{\cpx}(h) =
\{ x \in \cpx^n:\, h(x)=0  \}, \quad
V_{\re}(h) = V_{\cpx}(h) \cap \re^n.
\ee
When the complex variety $V_{\cpx}(h)$ is a finite set,
Laurent \cite{Lau07} proved that
$\{f_k\}$ has finite convergence to $f_{min}$.
When the real variety $V_{\re}(h)$ is a finite set,
Laurent \cite[Theorem~6.15]{Lau} proved that
$\{f_k^*\}$ has finite convergence to $f_{min}$.
In the case that $V_{\re}(h)$ is finite
but $V_{\cpx}(h)$ is infinite, it was unknown whether
$\{f_k\}$ has finite convergence to $f_{min}$ or not.
Indeed, Laurent \cite[Question~6.17]{Lau} asked:
\begin{center}
\begin{quote}
Does there exist an example with $|V_{\cpx}(h)|=\infty$,
$|V_{\re}(h)|<\infty$ and where
$f_k < f_{min}$ for all $k$?
\end{quote}
\end{center}
This question was also asked by Laurent
in the workshop {\sl Positive Polynomials and Optimization} (Banff, Canada, 2006),
and remained open since then, in the author's best knowledge.
Semidefinite relaxations are very useful for solving
zero-dimensional polynomial systems. We refer to \cite{LLR08,LR12}.

Our first main result is to give a negative answer
to the above question.
We prove that if $V_{\re}(h)$ is finite then $f_k=f_{min}$ for all $k$ big enough,
no matter $V_{\cpx}(h)$ is finite or not.
This is summarized as follows.

\begin{theorem}  \label{thm:fin-reVar}
Let $f_k,f_{min}$ be as above.
If the real variety $V_{\re}(h)$ is finite,
then $f_k=f_{min}$ for all $k$ big enough.
\end{theorem}

When $V_{\re}(h)$ is finite,
Theorem~\ref{thm:fin-reVar} implies that
there is no duality gap between \reff{sos:Put} and \reff{mom-las:put},
i.e., $f_k - f_k^* =0$, for $k$ big enough, because $f_k \leq f_k^* \leq f_{min}$.
This is a nice property for numerical computations.
When primal-dual interior point methods are applied to
solve semidefinite programs like \reff{sos:Put}-\reff{mom-las:put},
zero duality gap is often required.

The real variety $V_{\re}(h)$ can be defined
by different sets of polynomials, e.g.,
it can be defined by a single equation like
\[
 h_1^2(x)+\cdots+h_{m_1}^2(x)=0.
 \]
Suppose $h^{\prm} = (h_1^{\prm}, \ldots, h_r^{\prm})$ is a different tuple of
polynomials such that $V_{\re}(h^{\prm}) = V_{\re}(h)$.
Then, \reff{pop:gen} is equivalent to
\be \label{pop:h-prm}
\min \quad f(x) \quad s.t. \quad h^{\prm}(x) = 0, \quad g(x) \geq  0.
\ee
Like $\langle h\rangle_{2k}$,
we similarly define the truncated ideal $\langle h^{\prm}\rangle_{2k}$.
Then, Lasserre's hierarchy for \reff{pop:h-prm} is
the sequence of SOS relaxations ($k\in \N$)
\be  \label{las:h-prm}
f_k^{\prm}:=\max \quad \gamma \quad
s.t. \quad f - \gamma \in \langle h^{\prm} \rangle_{2k} + Q_k(g).
\ee
Similarly, we have $f_k^{\prm} \leq f_{min}$ for all $k$.
The following two questions are natural
about the two sequences $\{f_k\}$ and $\{f_k^{\prm}\}$:
\bit

\item If $\{f_k\}$ has finite convergence to $f_{min}$, does $\{f_k^{\prm}\}$
necessarily have finite convergence to $f_{min}$?

\item If $\{f_k\}$ has no finite convergence to $f_{min}$, is it possible
that $\{f_k^{\prm}\}$ has finite convergence to $f_{min}$?

\eit
%
%

When the real variety $V_{\re}(h)$ is finite, by Theorem~\ref{thm:fin-reVar},
the above two questions are solved:
the finite convergence of Lasserre's hierarchy is independent of
the choice of defining polynomials for $V_{\re}(h)$.
When $V_{\re}(h)$ is infinite, do we have a similar result?
Indeed, this is true under a general condition
on $V_{\re}(h)$ and the feasible set $K$ of \reff{pop:gen}.
The {\it vanishing ideal} of $K$ is defined as
\[
I(K):=\{p\in \re[x]:\, p(u) = 0 \, \forall \,
u \in K \}.
\]
The vanishing ideal of the real variety $V_{\re}(h)$ is
\[
I(V_{\re}(h)):=\{p\in \re[x]:\, p(u) = 0 \, \forall \,
u \in V_{\re}(h) \}.
\]
It is also called the
{\it real radical} of $\langle h \rangle$ (cf.~\cite{BCR}).

Our second main result is the following theorem.

\begin{theorem}  \label{thm:cvg-ne-h}
Let $h^{\prm} = (h_1^{\prm}, \ldots, h_r^{\prm})$
be a tuple of polynomials in $\re[x]$ such that
$V_{\re}(h)=V_{\re}(h^{\prm})$,
and $f_k,f_k^{\prm},f_{min}$ be defined as above.
Suppose $I(K) =I(V_{\re}(h))$.
Then, the sequence $\{f_k\}$ has finite convergence to $f_{min}$
if and only if $\{f_k^{\prm}\}$ has finite convergence to $f_{min}$.
\end{theorem}

In Theorem~\ref{thm:cvg-ne-h}, the condition $I(K) =I(V_{\re}(h))$
implies that if a polynomial $p$ identically vanishes on $K$
then it also identically vanishes on $V_{\re}(h)$.
It essentially requires that the feasible set $K$
and the real variety $V_{\re}(h)$ have the same Zariski closure.
This is often satisfied.

We would like to remark that there does not exist a similar
result like Theorem~\ref{thm:cvg-ne-h} for the case of inequalities.
That is, the choice of inequality constraining polynomials
might affect finite convergence of Lasserre's hierarchy,
while the feasible set $K$ is not changed.
For instance, consider the problem
\[
\min_{x \in \re} \quad 1-x^2 \quad s.t. \quad 1-x^2  \geq 0.
\]
Clearly, Lasserre's hierarchy for the above converges in one step,
and the problem is equivalent to
\[
\min_{ x \in \re } \quad 1-x^2 \quad s.t. \quad (1-x^2)^3 \geq 0.
\]
However, Lasserre's sequence $\{f_k\}$ for
the above new formulation does not have finite convergence.
Indeed, there exists a constant $C>0$ such that $f_k \leq -C k^{-2}$ for all $k$.
This is implied by Stengle~\cite[Theorem~4]{Stg96}.

This paper is organized as follows.
Section~\ref{sec:fin-var} is mostly to prove Theorem~\ref{thm:fin-reVar};
Section~\ref{sec:gen-var} is mostly to prove Theorem~\ref{thm:cvg-ne-h};
Section~\ref{sec:K-fin} proves that
if only the feasible set $K$ is finite,
then a refined version of Lasserre's hierarchy
(using the preordering of $g$) has finite convergence.

\section{Optimization with finite real varieties} \label{sec:fin-var}
\setcounter{equation}{0}

This section is mostly to prove Theorem~\ref{thm:fin-reVar}.
We begin with a useful lemma.

\begin{lemma} \label{lm:real-ideal}
(i) Let $\ell \geq 1$ be an integer. Then, for all
\[
c\geq  c_0:=\frac{1}{ 2\ell }
\left(1 - \frac{1}{2\ell} \right)^{2\ell-1},
\]
the univariate polynomial $s_c(t):=1+t+c t^{2\ell}$ in $t$ is SOS. \\
(ii) Let $p,q \in \re[x]$ and $\ell \geq 1$ be an integer.
Then, for all $\eps >0$ and $c \in \re$,
\[
p + \eps = \phi_\eps + \theta_\eps,
\]
where
\[
\phi_\eps = - c \eps^{1-2\ell} ( p^{2\ell} + q ), \quad
\theta_\eps = \eps s_c( p/\eps ) + c \eps^{1-2\ell} q.
\]
(iii) In (ii), assume $c \geq c_0$ as in (i),
$p^{2\ell}+q \in \langle h \rangle$
and $q \in Q(g)$ for polynomial tuples $h,g$.
Then, there exists an integer $N>0$ such that, for all $\eps>0$,
\[
\phi_\eps \in \langle h \rangle_{2N}, \quad \theta_\eps \in Q_N(g).
\]
\end{lemma}
\begin{proof}
(i) For all $c>0$, the univariate polynomial $s_c(t)$
is convex in $t$ over the real line $\re$
and $s_c^{\prm}(t)= 1 + 2\ell c t^{2\ell-1}$.
The polynomial $s_c$ has a unique real critical point
$\xi:=\left(\frac{-1}{2\ell c}\right)^{\frac{1}{2\ell-1}}$.
Note that
\[
s_c(\xi) = 1 + \left(\frac{-1}{2\ell c}\right)^{\frac{1}{2\ell-1}}
\left(
1 -  \frac{1}{2\ell}
\right).
\]
It can be verified that $s_c(\xi) \geq 0$
if and only if $c\geq c_0$. So, when $c\geq c_0$,
the univariate polynomial $s_c$ is nonnegative over $\re$
(because $s_c(\xi)\geq 0$, $s_c^{\prm}(\xi) = 0$ and $s_c$ is convex),
and it must be SOS (cf.~\cite{Rez00}).

(ii) It can be done by a direct verification.

(iii) By assumption, there exist positive integers $N_1, N_2$ such that
$p^{2\ell}+q \in \langle h \rangle_{2N_1}$, $q \in Q_{N_2}(g)$.
Let $N_0 = \ell \lceil \deg(p)/2 \rceil$.
Note that $s_c(p/\eps)$ is SOS by (i) and its degree is at most $2N_0$.
So, $\eps s_c(p/\eps) \in Q_{N_0}(g)$ for all $\eps>0$.
Then $N:=\max(N_0,N_1,N_2)$ works for the proof.
\end{proof}

Theorem~\ref{thm:fin-reVar} can be proved by using Lemma~\ref{lm:real-ideal}.

\begin{proof}[Proof of Theorem~\ref{thm:fin-reVar}]
When $V_{\re}(h)$ is empty, the feasible set $K$ is also empty,
and hence $f_{min} = + \infty$ by convention.
By Positivstellensatz (cf.~\cite[Theorem~4.4.2]{BCR}),
we have $-1 \in \langle h \rangle + \Sig \re[x]^2$.
For all $\gamma >0$, it holds that
\[
f - \gamma = (1+f/4)^2 + (-1) ( \gamma + (1-f/4)^2) \in
\langle h \rangle_{2k} + Q_k(g),
\]
for all $k$ big enough. So, for all big $k$, \reff{sos:Put}
is unbounded from above, and hence $f_k = +\infty$.
Hence, Lasserre's hierarchy has finite convergence.

When $V_{\re}(h)$ is nonempty and finite, we can write $V_{\re}(h)=\{u_1, \ldots, u_D\}$
for distinct points $u_1,\ldots,u_D \in \re^n$.
Let $\varphi_1, \ldots, \varphi_D \in \re[x]$
be the interpolating polynomials such that
$\varphi_i(u_j) = 0$ for $i \ne j$ and $\varphi_i(u_j) = 1$ for $i = j$.
For each $u_i$, if $f(u_i) - f_{min} \geq 0$,
let $a_i := (f(u_i) - f_{min}) \varphi_i^2$.
If $f(u_i) - f_{min} < 0$, then at least one of
$g_1(u_i),\ldots, g_{m_2}(u_i)$ is negative, say, $g_{j_i}(u_i) < 0$,
and let
\[
a_i := \left(\frac{f(u_i) - f_{min}}{g_{j_i}(u_i)}\right)
g_{j_i} \varphi_i^2.
\]
Each $a_i$ is a polynomial in $Q(g)$.
Let $a := a_1+ \cdots + a_D$.
By construction, $a \in Q_{N_1}(g)$ for some integer $N_1 > 0$.
The polynomial
\[
\hat{f}:=f - f_{min} - a
\]
vanishes identically on $V_{\re}(h)$.
By Real Nullstellensatz (cf.~\cite[Corollary~4.1.8]{BCR}),
there exist an integer $\ell>0$ and $q \in \Sig\re[x]^2$ such that
\[
\hat{f}^{2\ell} + q \in  \langle h \rangle.
\]
Apply Lemma~\ref{lm:real-ideal} to
$p := \hat{f}, q$, with the tuples $h,g$ and any $c \geq \frac{1}{2\ell}$.
Then, there exists $N \geq N_1$ such that, for all $\eps >0$,
\[
\hat{f} + \eps = \phi_\eps + \theta_\eps,
\]
and $\phi_\eps \in \langle h \rangle_{2N}$, $\theta_\eps \in Q_N(g)$.
Therefore, we get
\[
f - (f_{min} -\eps) = \phi_\eps + \sig_\eps,
\]
where $\sig_\eps = \theta_\eps +a \in Q_N(g)$ for all $\eps >0$.
This implies that, for all $\eps>0$,
$\gamma = f_{min}-\eps$ is feasible in \reff{sos:Put} for the order $N$.
Thus, we get $f_N \geq f_{min}$.
Note that $f_k \leq f_{min}$ for all $k$ and $\{f_k\}$
is monotonically increasing.
So, we must have $f_k = f_{min}$ for all $k \geq N$,
i.e., Lasserre's hierarchy has finite convergence.
\end{proof}

We present some examples to show the proof of Theorem~\ref{thm:fin-reVar}.

\begin{exm}
Consider the optimization problem
\[
\left\{ \baray{rl}
\min &  f(x):=x_1x_2 \\
 s.t. & h(x):=(x_1^2-1)^2+(x_2^2-1)^2=0, \\
 &  g(x):= x_1+x_2-1 \geq 0.
\earay \right.
\]
Clearly, $V_{\re}(h) =\{(\pm 1, \pm 1)\}$,
$K=\{(1,1)\}$ and $f_{min}=1$. Let
\[
a = \half (x_1+x_2-1) (x_1-x_2)^2 \in Q_2(g),
\]
\[
\hat{f} = f-1-a=
\half
\left[(x_2^2-1)(x_1-x_2+1) - (x_1^2-1)(x_1-x_2-1) \right].
\]
Then, $\hat{f} \equiv 0$ on $V_{\re}(h)$ and
\[
\hat{f}^2 + q =
\half ( (x_1-x_2)^2 + 1 ) h
\in \langle h \rangle_6,
\]
where
\[
q = \frac{1}{4}
\left( (x_1^2-1)(x_1-x_2+1) + (x_2^2-1)(x_1-x_2-1)  \right)^2.
\]
For each $\eps>0$, let
\[
\phi_\eps  = -  \frac{1}{4\eps}
(\hat{f}^2+ q) \in \langle h \rangle_6, \quad
\sig_\eps = \eps \left(1+ \frac{\hat{f}}{2\eps} \right)^2
+ \frac{1}{4\eps} q + a \in Q_3(g).
\]
Then, $f - 1 + \eps = \phi_\eps + \sig_\eps$ for all $\eps>0$.
So, $f_k =1$ for all $k \geq 3$.
\qed
\end{exm}

\begin{exm} \label{emp:sum:h2d=0}
Let $f\in \re[x]$ be such that $f(0)=0$. Consider the problem
\[
\left\{ \baray{rl}
\min & f(x) \\
s.t. &
h(x):=x_1^{2d} + \cdots + x_n^{2d} = 0.
\earay \right.
\]
Clearly, $f_{min}=0$. There are no inequality constraints,
and we can think that $g=0$, as in \reff{pop:gen}. Write $f$ as
\[
f = x_1b_1 + \cdots + x_nb_n, \quad
b_1, \ldots, b_n \in \re[x].
\]
Let $\Sig_{n,2d}$ be the cone of SOS forms in $n$ variables and of degree $2d$.
There exists $\lmd >0$ such that
\[
\lmd (t_1^{2d}+\cdots+t_n^{2d}) -
(t_1+\cdots+t_n)^{2d} \in \Sig_{n,2d}.
\]
This is because $t_1^{2d}+\cdots+t_n^{2d}$ lies in the interior
of $\Sig_{n,2d}$ (cf. \cite[Proposition~5.3]{Mar09}).
By replacing each $t_i$ by $x_ib_i$ in the above, we know that
\[
\psi: =
\lmd((x_1b_1)^{2d}+\cdots+(x_nb_n)^{2d}) - f^{2d} \in \Sig\re[x]^2.
\]
Clearly, it holds that
\[
\eta := \lmd
\left[\left( \sum_{i=1}^n x_i^{2d} \right)\left(\sum_{i=1}^n b_i^{2d}\right)
- \left( \sum_{i=1}^n (x_ib_i)^{2d} \right) \right] \in \Sig\re[x]^2,
\]
\[
f^{2d} + \psi + \eta =
\lmd (x_1^{2d}+\cdots+x_n^{2d})(b_1^{2d}+\cdots+b_n^{2d})
\in \langle h \rangle.
\]
Let $q := \psi+\eta \in \Sig\re[x]^2$.
Clearly, $f \equiv 0$ on $V_{\re}(h)$, and $f^{2d} + q \in \langle h \rangle$.
Suppose $\deg(f) = r$. Apply Lemma~\ref{lm:real-ideal} with $c = \frac{1}{2d}$,
$\ell = d$ and $p =f$. For each $\eps>0$, let
\[
\phi_\eps = - \frac{1}{2d} \eps^{1-2d} (f^{2d}+q)
\in \langle h \rangle_{2dr},
\]
\[
\sig_\eps = \eps \left(1+f/\eps+ \frac{1}{2d} (f/\eps)^{2d} \right)
+ \frac{1}{2d} \eps^{1-2d} q \in Q_{dr}(0).
\]
Then, $f+\eps = \sig_\eps  + \phi_\eps$ for all $\eps>0$.
So, $f_k =0$ for all $k\geq dr$.
\qed
\end{exm}

We would like to remark that the SOS relaxation \reff{sos:Put}
might not achieve its optimal value $f_k$ for any order $k$,
even if $\{f_k\}$ has finite convergence to $f_{min}$.
For instance, consider the problem
\[
\min \quad x_1 \quad s.t. \quad  x_1^2+x_2^2+\cdots+x_n^2= 0.
\]
By Example~\ref{emp:sum:h2d=0}, we know
$f_{k} = 0$ for all $k\geq 1$. However,
for any $\phi \in \re[x]$, the polynomial
$
\varphi =x_1 - (x_1^2+x_2^2+\cdots+x_n^2)\phi
$
cannot be SOS (because $\varphi(0)=0, \nabla \varphi(0) \ne 0$,
and $0$ can not be a minimizer of $\varphi$).
For this problem, \reff{sos:Put} does not have a maximizer,
for any order $k\geq 1$.

However, in Theorem~\ref{thm:fin-reVar}, if
the ideal $\langle h \rangle$ is {\it real},
i.e., $\langle h \rangle = I(V_{\re}(h))$ (cf. \cite[Section~4.1]{BCR}),
then \reff{sos:Put} achieves its optimum for all big $k$.

\begin{prop} \label{fiV:sos-opt}
In Theorem~\ref{thm:fin-reVar}, if, in addition,
the ideal $\langle h \rangle$ is {\it real},
then \reff{sos:Put} achieves its optimum for all $k$ big enough.
\end{prop}
\begin{proof}
Let $a$ be from the proof of Theorem~\ref{thm:fin-reVar}.
We know that $\hat{f} = f-f_{min}-a$ identically vanishes on
$V_{\re}(h)$. So, $\hat{f} \in I(V_{\re}(h))$.
Since $\langle h \rangle$ is real, $I(V_{\re}(h)) = \langle h \rangle$
and $\hat{f} \in \langle h \rangle$.
The identity $f-f_{min} = a + \hat{f}$
implies that $\gamma = f_{min}$ is feasible in \reff{sos:Put}
if $k$ is big enough. Thus, \reff{sos:Put} achieves its optimum
$f_{min}$ for all big $k$.
\end{proof}

When $V_{\re}(h)$ is not finite,
the conclusion of Proposition~\ref{fiV:sos-opt}
also holds under some other conditions.

\begin{prop} \label{sos:achv:opt}
Let $h$ and $K$ be as in \reff{pop:gen}.
Suppose that $f_{min}$ is finite and Lasserre's hierarchy has finite convergence.
If $\langle h \rangle = I(K)$,
then \reff{sos:Put} achieves its optimum for all $k$ big enough.
\end{prop}
\begin{proof}
There exists $N_1$ such that $f_k = f_{min}$ for all $k\geq N_1$.
By the condition that $I(K)=\langle h \rangle$,
we know the quotient set $Q_k(g)/\langle h \rangle$ is closed for all $k$
(cf. Laurent~\cite[Theorem~3.35]{Lau} or Marshall~\cite[Theorem~3.1]{Mar03}).
Let $\{\gamma_i\}_{i=1}^\infty$ be a sequence such that
each $\gamma_i$ is feasible for \reff{sos:Put}
with $k=N_1$ and $\gamma_i \to f_{min}$ as $i \to \infty$.
Clearly, each $f-\gamma_i \in Q_{N_1}(g)/\langle h \rangle$
and $f-\gamma_i \to f - f_{min}$.
Hence, $f-f_{min} \in Q_{N_1}(g)/\langle h \rangle$, i.e.,
there exists $\phi^* \in  \langle h \rangle$ and $\sig^* \in Q_{N_1}(g)$
such that
\[
f-f_{min} =  \phi^* + \sig^*.
\]
Let $N_2 \geq N_1$ be such that $\phi^* \in \langle h \rangle_{2N_2}$.
Then, $\gamma = f_{min}, \phi^*, \sig^*)$ is feasible for \reff{sos:Put}
with order $k \geq N_2$.
Hence, \reff{sos:Put} achieves its optimum for all $k\geq N_2$.
\end{proof}

\section{Optimization with general real varieties}  \label{sec:gen-var}
\setcounter{equation}{0}

This section is mostly to prove Theorem~\ref{thm:cvg-ne-h}.
We first prove a result that similar to Theorem~\ref{thm:cvg-ne-h}
by using generators of the real radical $I(V_{\re}(h))$.

Let $h^{rad}_1,\ldots,h^{rad}_t$ be a set of generators for $I(V_{\re}(h))$, i.e.,
\[
I(V_{\re}(h)) = \langle h^{rad}_1,\ldots,h^{rad}_t \rangle.
\]
Denote $h^{rad}:=(h^{rad}_1,\ldots,h^{rad}_t)$.
Define $\langle h^{rad} \rangle_{2k}$ similarly as for $\langle h \rangle_{2k}$.
Clearly, \reff{pop:gen} is equivalent to
\be  \label{pop:rad-poly}
\left\{\baray{rl}
\min & f(x) \\
s.t. & h^{rad}_i(x) = 0 \,(i=1,\ldots,t),  \\
 & g_j(x) \geq  0 \,(j=1,\ldots,m_2).
\earay \right.
\ee
Lasserre's hierarchy for \reff{pop:rad-poly} is
the sequence of SOS relaxations ($k\in \N$)
\be  \label{sos:re-rad}
f^{rad}_k := \, \max   \quad \gamma \quad
s.t. \quad f - \gamma  \in \langle h^{rad} \rangle_{2k} + Q_k(g).
\ee
We also have $f^{rad}_k \leq f_{min}$ for all $k$.

\begin{theorem}  \label{fincvg=Rad}
Let $h$, $f_{min}$ and $K$ be as in \reff{pop:gen}.
Suppose that $f_{min}$ is finite and $I(K)=I(V_{\re}(h)) = \langle h^{rad}\rangle$.
Let $f_k$ (resp., $f^{rad}_k$) be the optimal value of
\reff{sos:Put} (resp., \reff{sos:re-rad}).
Then, the sequence $\{f_k\}$ has finite convergence to $f_{min}$
if and only if $\{f^{rad}_k\}$ has finite convergence to $f_{min}$.
\end{theorem}
\begin{proof}
First, assume that $\{f^{rad}_k\}$ has finite convergence to $f_{min}$.
The feasible set of \reff{pop:rad-poly} is $K$
and $\langle h^{rad}\rangle = I(K)$.
%
%
%
Apply Proposition~\ref{sos:achv:opt} to Lasserre's sequence
$\{f^{rad}_k\}$ for \reff{pop:rad-poly} with the tuple $h^{rad}$.
We know that \reff{sos:re-rad} achieves its optimum $f_{min}$
for all big $k$, say, for all $k\geq N_1$.
Let $p \in  \langle h^{rad} \rangle_{2N_1}$ and $\sig_1 \in Q_{N_1}(g)$
be such that
\[
f-f_{min} =  p + \sig_1.
\]
Since $\langle h^{rad} \rangle = I(V_{\re}(h))$,
$p \equiv 0$ on $V_{\re}(h)$.
By Real Nullstellensatz (cf.~\cite[Corollary~4.1.8]{BCR}),
there exist an integer $\ell>0$ and $q\in \Sig\re[x]^2$ such that
\[
p^{2\ell} + q  \, \in  \, \langle h \rangle.
\]
By Lemma~\ref{lm:real-ideal},
there exists $N_2>0$ such that, for all $\eps >0$,
\[
p + \eps = \phi_\eps + \theta_\eps,
\]
with $\phi_\eps \in \langle h \rangle_{2N_2}$,
$\theta_\eps \in Q_{N_2}(g)$.
Let $\sig_\eps = \theta_\eps + \sig_1$ and $N_3 = \max(N_1,N_2)$. Then,
\[
f-(f_{min}-\eps) = \sig_\eps  + \phi_\eps, \quad
\sig_\eps \in Q_{N_3}(g), \quad \phi_\eps \in \langle h \rangle_{2N_3}.
\]
Hence, $f_k = f_{min}$ for all $k\geq N_3$, i.e.,
$\{f_k\}$ has finite convergence to $f_{min}$.

\medskip

Second, assume that $\{f_k\}$ has finite convergence to $f_{min}$,
say, $f_k = f_{min}$ for all $k\geq M_1$.
Thus, for every $\eps>0$, there exist
$\phi_\eps \in \langle h \rangle_{2M_1}$,
$\sig_\eps \in Q_{M_1}(g)$ such that
\[
f - (f_{min} - \eps) = \phi_\eps + \sig_\eps.
\]
Note that each $h_i \in I(V_{\re}(h)) = \langle h^{rad} \rangle$.
So, there exists $M_2 \geq M_1$ such that
$\langle h \rangle_{2M_1} \subseteq \langle h^{rad} \rangle_{2M_2}$
and $Q_{M_1}(g) \subseteq Q_{M_2}(g)$.
This implies that $f^{rad}_k \geq f_{min}-\eps$
for all $k\geq M_2$ and for all $\eps >0$.
Hence, $f^{rad}_k \geq f_{min}$ for all $k\geq M_2$.
Since $f^{rad}_k \leq f_{min}$ for all $k$,
we know that $\{f^{rad}_k\}$ has finite convergence to $f_{min}$.
%
%
\end{proof}

Theorem~\ref{thm:cvg-ne-h} can be proved by using Theorem~\ref{fincvg=Rad}.

\begin{proof}[Proof of Theorem~\ref{thm:cvg-ne-h}]
If $f_{min}=-\infty$, then $f_k,f_k^{\prm} \leq f_{min} = -\infty$ for all $k$,
and the conclusion of Theorem~\ref{thm:cvg-ne-h} is clearly true.
If $f_{min}=+\infty$, then $K =\emptyset$ and $I(K) = \re[x]$;
so, $I(V_{\re}(h))=I(K) =\re[x]$ and $V_{\re}(h)=\emptyset$.
The conclusion of Theorem~\ref{thm:cvg-ne-h} is also true,
as shown at the beginning of the proof of Theorem~\ref{thm:fin-reVar}.
%
%
%
%

Now we prove Theorem~\ref{thm:cvg-ne-h} when $f_{min}$ is finite.
Let $h^{rad}$ and $f_k^{rad}$ be as in Theorem~\ref{fincvg=Rad}.
By Theorem~\ref{fincvg=Rad},
$\{f_k\}$ has finite convergence to $f_{min}$
if and only if $\{f^{rad}_k\}$ has finite convergence to $f_{min}$.
For the same reason, since $V_{\re}(h) = V_{\re}(h^\prm)$ and
\reff{pop:gen} is equivalent to \reff{pop:h-prm},
$\{f_k^{\prm}\}$ has finite convergence to $f_{min}$
if and only if $\{f^{rad}_k\}$ has finite convergence to $f_{min}$.
This shows that $\{f_k\}$ has finite convergence to $f_{min}$
if and only if $\{f_k^{\prm}\}$ has finite convergence to $f_{min}$.
\end{proof}

A direct consequence of of Theorem~\ref{thm:cvg-ne-h}
is that we can reduce the number of
equality constraints in polynomial optimization,
while finite convergence of Lasserre's hierarchy is not lost.
As is well known, every real variety can be defined
by a single equation. Let
\[
h^{sq}(x):=h_1^2(x) + \cdots + h_{m_1}^2(x).
\]
Then, \reff{pop:gen} is equivalent to
\be  \label{pop:sqr-eq}
\min \quad f(x) \quad
s.t.  \quad h^{sq}(x) = 0,\, g(x) \geq 0.
\ee
Lasserre's hierarchy for \reff{pop:sqr-eq} is
the sequence of SOS relaxations ($k \in \N$)
\be  \label{las-sos:sqr}
f_k^{sq} := \max \quad \gamma \quad
s.t. \quad  f - \gamma  \in \langle h^{sq} \rangle_{2k} + Q_k(g).
\ee
If $I(K) = I(V_{\re}(h))$, by Theorem~\ref{thm:cvg-ne-h},
$\{f_k\}$ has finite convergence to $f_{min}$
if and only if $\{f^{sq}_k\}$ has finite convergence to $f_{min}$.
We show an example of this.

\begin{exm}
Consider the optimization problem:
\be \label{exm3:sqr}
\left\{ \baray{rl}
\min & f(x):=x_1x_2x_3-2x_3 \\
 s.t. & h^{sq}(x):=(x_1^2-x_2)^2+(x_1^3-x_3)^2=0.
\earay \right.
\ee
It has no inequality constraints, and we can think that $g=0$.
Its feasible set is the curve parameterized as $(x_1, x_1^2, x_1^3)$.
The minimum $f_{min}=-1$.
We show that the sequence $\{f^{sq}_k\}$ for \reff{exm3:sqr}
has finite convergence. Let $\sig_1 = (x_1^3-1)^2$ and
\[
p =  (x_1^3-2)(x_3-x_1^3) + x_1x_3(x_2-x_1^2).
\]
Then, $f + 1 = p + \sig_1$. Clearly, $p \equiv 0$ on $V_{\re}(h^{sq})$ and
\[
p^2 + q = h^{sq} \psi,
\]
where
\[
q = \Big(
x_1x_3(x_3-x_1^3) -(x_1^3-2)(x_2-x_1^2)
\Big)^2, \quad
\psi =  x_1^2x_3^2 +(x_1^3-2)^2.
\]
For all $\eps>0$, we have $f + 1 + \eps = \phi_\eps + \sig_\eps$ where
\[
\phi_\eps = \frac{-1}{4\eps} \psi h^{sq} \in \langle h^{sq} \rangle_{12}, \quad
\sig_\eps = \eps \left(1+ \frac{p}{2\eps}\right)^2
+ \frac{1}{4\eps} q + \sig_1 \in Q_6(0).
\]
So, $f_k^{sq} = -1$ for all $k \geq 6$.
\qed
\end{exm}

We show an application of Theorem~\ref{thm:cvg-ne-h}
in gradient SOS relaxations for minimizing polynomials \cite{NDS06}.
Consider the unconstrained optimization problem
\be  \label{uc:pop}
\min_{x\in \re^n } \quad f(x).
\ee
If \reff{uc:pop} has a minimizer, then it is equivalent to
\be  \label{pop:g(f)=0}
\min_{x\in \re^n } \quad f(x) \quad
s.t. \quad \nabla f(x) = 0.
\ee
When $\langle \nabla f \rangle$ is radical,
Lasserre's hierarchy for \reff{pop:g(f)=0} has finite convergence \cite{NDS06}.
Indeed, the finite convergence also occurs even if
$\langle \nabla f \rangle$ is not radical, as shown in \cite{Nie-jac}.
Clearly, \reff{pop:g(f)=0} is equivalent to
\be  \label{pop:||g(f)||^2=0}
\min_{x\in \re^n } \quad f(x) \quad
s.t. \quad  \|\nabla f(x)\|_2^2 = 0.
\ee
An advantage of \reff{pop:||g(f)||^2=0} over \reff{pop:g(f)=0}
is that  \reff{pop:||g(f)||^2=0} has a single equality constraint.
By Theorem~\ref{thm:cvg-ne-h},
Lasserre's hierarchy of \reff{pop:||g(f)||^2=0} also has finite convergence.

\begin{exm} (\cite{Las01,NDS06})
Consider the polynomial optimization problem
\[
\min_{x\in\re^2} \quad f(x):=x_1^2x_2^2(x_1^2+x_2^2-1).
\]
The minimum $f_{min}=-1/27$ is
achieved at $(\pm 1, \pm 1)/\sqrt{3}$. We have
\[
\pt f/\pt x_1 = 2x_1x_2^2(2x_1^2+x_2^2-1), \quad
\pt f/\pt x_2 =  2x_1^2x_2(x_1^2+2x_2^2-1).
\]
This optimization problem is equivalent to
\[
\min_{x\in\re^2} \quad f(x) \quad s.t. \quad  \| \nabla f(x) \|_2^2 = 0,
\]
where $\| \nabla f(x) \|_2^2$ has the representation
\[
4 x_1^2x_2^2 \Big(
x_2^2(2x_1^2+x_2^2-1)^2 + x_1^2(x_1^2+2x_2^2-1)^2
\Big).
\]
Let $\sig_1 =  3(x_1^2x_2^2-1/9)^2$ and $\hat{f}= f +1/27 - \sig_1$.
Then, $\hat{f} \equiv 0$ on $V_{\re}(\nabla f)$ and
\[
\hat{f}^2 + q = \| \nabla f \|_2^2 \psi,
\]
where $q, \psi$ are SOS polynomials given as
\[
a=x_1^2-1/3, \quad b=x_2^2-1/3,
\]
\[
s_1 = x_1^2x_2^2\Big( x_2^4(2a+b)^2 + x_1^4(a+2b)^2 \Big), \quad
s_2 = 4(a+b)^2(a^2+b^2)+a^4+b^4,
\]
\[
q = \frac{9}{2}(s_1 (a^2+b^2) + x_1^4x_2^4 s_2), \quad
\psi = \frac{9}{8}(x_1^2+x_2^2)(a^2+b^2).
\]
For each $\eps>0$, let
\[
\phi_\eps = -\frac{1}{4\eps} \| \nabla f \|_2^2 \psi
\in \langle \| \nabla f \|_2^2 \rangle_{16}, \quad
\sig_\eps = \eps \left(1+ \hat{f}/2\eps\right)^2
+ \frac{q}{4\eps} + \sig_1 \in Q_8(0).
\]
Then, $f +1/27 +\eps = \sig_\eps + \phi_\eps$ for all $\eps>0$.
So, $f_k^{grad} = -1/27$ for all $k \geq 8$.
\qed
\end{exm}

\section{Optimization over finite semialgebraic sets}  \label{sec:K-fin}
\setcounter{equation}{0}

In this section, we consider the case that
the feasible set $K$ of \reff{pop:gen} is a finite set
while the real variety $V_{\re}(h)$ is not necessarily.
To apply Theorem~\ref{thm:fin-reVar},
a natural idea is to introduce new variables $z_1, \ldots, z_{m_2}$.
Then, $K$ can be equivalently defined by the equations
\[
h(x) = 0, \quad
g_1(x)-z_1^2 = \cdots = g_{m_2}(x)-z_{m_2}^2 = 0.
\]
Clearly, $K$ is a finite set if and only if the above equations
have finitely many real solutions. If $K$ is finite,
by Theorem~\ref{thm:fin-reVar}, Lasserre's hierarchy has finite convergence
if we use the above equivalent polynomial equalities
in both $x_1,\ldots,x_n$ and $z_1, \ldots, z_{m_2}$.
However, this approach introduces new variables $z_1, \ldots, z_{m_2}$,
which typically make the resulting SOS relaxations very difficult to solve.
To get a finitely convergent hiearchy of SOS relaxations
that only uses the original polynomials in $x$,
we need stronger relaxations than \reff{sos:Put}.

Let $Pr_k(g)$ be the $k$-th truncated quadratic module
generated by the set of all possible cross products:
\[
g_1, \,\, \ldots, \,\, g_{m_2}, \, g_1g_2, \,\, \ldots, \,\, g_{m_1-1}g_{m_1},\, \,
\ldots, \,\, g_1g_2\cdots g_{m_2}.
\]
The set $Pr_k(g)$ is also called the $k$-th {\it truncated preordering}
generated by $g=(g_1,\ldots, g_{m_2})$ (cf.~\cite{LasBok,Lau,MarBk}).
Consider the sequence of SOS relaxations ($k\in \N$)
\be  \label{sos:preord}
f_k^{pre} := \max  \quad \gamma \quad
s.t. \quad f - \gamma  \in \langle h \rangle_{2k} + Pr_k(g).
\ee
If $K$ is compact, then $\{f_k^{pre}\}$ asymptotically converges
to $f_{min}$ (cf.~\cite{Las01,Smg}).
When $K$ is finite, the sequence of optimal values of
the dual problem of \reff{sos:preord} has finite convergence,
as shown by Lasserre, Laurent, and Rostalski \cite[Remark~4.9]{LLR08}.
Here, we show that the same result holds for the sequence $\{f_k^{pre}\}$.

\begin{theorem}  \label{thm:finset-pre}
Let $f_k^{pre}, f_{min}$ be as above.
If the feasible set $K$ of \reff{pop:gen} is finite,
then the sequence $\{f_k^{pre}\}$ has finite convergence to $f_{min}$.
\end{theorem}
\begin{proof}
The set $K$ consists of finitely many points, say, $u_1, \ldots, u_D \in \re^n$.
Let $\varphi_1, \ldots, \varphi_D \in \re[x]$ be the interpolating polynomials
such that $\varphi_i(u_j)=0$ for $i \ne j$
and $\varphi_i(u_j)=1$ for $i = j$.
Then, $f(u_i) - f_{min} \geq 0$ for all $i$. Let
\[
a := \sum_{i=1}^D (f(u_i) - f_{min}) \varphi_i^2 \in \Sig\re[x]^2.
\]
The polynomial $\hat{f}:=f - f_{min} -a$ vanishes identically on $K$.
By Positivstellensatz (cf.~\cite[Corollary~4.4.3]{BCR}),
there exist integers $\ell>0$ and $N_1>0$ such that
\[
q \in Pr_{N_1}(g), \qquad
\hat{f}^{2\ell} + q  \in \langle h \rangle_{2N_1}.
\]
Applying Lemma~\ref{lm:real-ideal} with $c \geq \frac{1}{2\ell}$ and $p=\hat{f}$,
we get that, for all $\eps >0$,
\[
f-(f_{min}-\eps) = p+\eps+ a = \sig_\eps + \phi_\eps,
\]
\[
\phi_\eps
= -c \eps^{1-2\ell}( \hat{f}^{2\ell} + q ) \in \langle h \rangle_{2N_1},
\]
\[
\sig_\eps = \eps \left(1+\hat{f}/\eps+ c (\hat{f}/\eps)^{2\ell} \right)
+ c \eps^{1-2\ell}q + a.
\]
Let $N \geq N_1$ be such that
$\sig_\eps \in Pr_{N}(g)$ for all $\eps >0$.
Like before, we have $f_k^{pre} = f_{min}$ for all $k\geq N$.
\end{proof}

We illustrate the proof of Theorem~\ref{thm:finset-pre}
with the following example.

\begin{exm}
Consider the optimization problem
\[
\baray{rl}
\min & -x_1^2-x_2^2  \\
s.t. & x_1^3 \geq 0, x_2^3 \geq 0,  -x_1-x_2-x_1x_2 \geq 0.
\earay
\]
Let $f,g_1,g_2,g_3$ be the objective,
the first, second and third constraining polynomials respectively.
Clearly, $K=\{(0,0)\}$ and $f_{min}=0$.
We have $f \equiv 0$ on $K$ and
\[ f^4 + q = 0,\]
where
\[
q =  \sig_0 + g_1 \sig_1 + g_2 \sig_2 +
g_1g_2 \sig_{12} + g_3 \sig_3.
\]
In the above, the SOS polynomials $\sig_0, \sig_1, \sig_2, \sig_{12}, \sig_3$ are given as:
{\smaller
%
%
\[
\sig_0 = (x_1^2-x_2^2)^4+6(x_1^4-x_2^4)^2,  \,
\sig_{12} = 32(x_1^2+x_2^2+x_1^4+x_2^4+x_1^6+x_2^6),
\]
\[
\baray{r}
\sig_1 = 8\Big(   x_1^6(x_2+1/2+4x_2^2) +
x_1^4(x_1^2/2+2x_1x_2+2x_2^2) + \qquad  \\
x_1^4(2x_2+1/2+2x_2^2) + x_1^2(x_1^2/2+x_1x_2+4x_2^2) + 4(x_2^4+x_2^6+x_2^8) \Big),
\earay
\]
\[
\baray{r}
\sig_2 = 8\Big( x_2^6(x_1+1/2+4x_1^2) +
x_2^4(x_2^2/2+2x_2x_1+2x_1^2) + \qquad \\
x_2^4(2x_1+1/2+2x_1^2) + x_2^2(x_2^2/2+x_2x_1+4x_1^2) + 4(x_1^4+x_1^6+x_1^8) \Big),
\earay
\]
\[
\sig_3 = 8\Big( (x_1^8+x_2^8+x_1^7+x_2^7 + x_1^6 + x_2^6)
+4x_1^2x_2^2(x_1^2+x_2^2+x_1^4+x_2^4+x_1^6+x_2^6) \Big) .
\]
}
Apply Lemma~\ref{lm:real-ideal} with $c=1/4$ and $p=f$.
For each $\eps>0$, let
\[
\phi_\eps =0, \quad
\sig_\eps := \eps \left(1+ \frac{f}{\eps} + \frac{f^4}{4\eps^4} \right)
+ \frac{1}{4\eps^3} q \in Pr_6(g).
\]
Then, $f+\eps = \phi_\eps + \sig_\eps$ for all $\eps>0$.
So, $f_k^{pre} = 0$ for all $k \geq 6$.
\qed
\end{exm}

\end{document}